\documentclass{amsart}
\usepackage{amssymb}
\usepackage{amsmath}
\usepackage{amsfonts}

\newtheorem{theorem}{Theorem}[section]
\theoremstyle{plain}

\newtheorem{corollary}[theorem]{Corollary}

\newtheorem{definition}[theorem]{Definition}
\newtheorem{example}[theorem]{Example}

\newtheorem{lemma}[theorem]{Lemma}

\newtheorem{proposition}[theorem]{Proposition}

\begin{document}
\title[]{Intra-regular Abel-Grassmann's groupoids}
\author{}
\maketitle

\begin{center}
\textbf{Wieslaw A. Dudek}

Institute of Mathematics and Computer Science

Wroclaw University of Technology, 50-370 Wroclaw, Poland

Wieslaw.Dudek@pwr.wroc.pl

\textbf{Madad Khan and Nasir Khan}

Department of Mathematics,

COMSATS Institute of Information Technology,

Abbottabad, Pakistan

madadmath@yahoo.com, nasirmaths@yahoo.com
\end{center}

\bigskip

\textbf{Abstract.} We characterize intra-regular Abel-Grassmann's groupoids by the properties of their ideals and $(\in ,\in\!\vee q_{k})$-fuzzy ideals of various types.


\bigskip

\section{Introduction}

\textit{Left almost semigroups} \cite{Pro2}, abbreviated as \textit{$LA$-semigroups}, are an algebraic structure midway between groupoids and commutative semigroups with wide applications in the theory of flocks \cite{Vag}. Certaine (cf. \cite{Cert}) applied idempotent flocks to affine geometry, as does Baer in his book \textit{Linear Algebra and Projective Geometry}.

$LA$-semigroups are also called \textit{Abel-Grssmann's groupoids} or \textit{$AG$-groupoids}. This structure is closely related with a commutative semigroup because if an $LA$-semi\-group contains right identity then it becomes a commutative monoid. An $LA$-semigroup with left identity is a semilattice \cite{MKh}. Although the structure is non-associative and non-commutative, nevertheless, it posses many interesting properties which we usually found in associative and commutative algebraic structures. For example, congruences of some $AG$-groupoids have very similar properties as congruences of semigroups (cf. \cite{DudGig} and \cite{Pro}). Moreover, any locally associative $AG$-groupoid $S$ with left identity is uniquely expressible as a semilattice of archimedean components. The archimedian components of $S$ are cancellative if and only if $S$ is separative. Such $AG$-groupoid can be embedded into a union of groups \cite{MIq}. On the other hand on some $AG$-groupoids one can define the structure of an abelian group \cite{Mus}.

Usually the models of real world problems in almost all disciplines like engineering, medical sciences, mathematics, physics, computer science, management sciences, operations research and artificial intelligence are mostly full of complexities and consist of several types of uncertainties while dealing them in several occasion. To overcome these difficulties of
uncertainties, many theories have been developed such as rough sets theory, probability theory, fuzzy sets theory, theory of vague sets, theory of soft ideals and the theory of intuitionistic fuzzy sets. Zadeh \cite{Zad} discovered the relationships of probability and fuzzy set theory which has appropriate approach to deal with uncertainties. Many authors have applied
the fuzzy set theory to generalize the basic theories of Algebra (cf. \cite{Bha2, DudSha, jun2, jun3, Ros, Shabir1}). Mordeson et al. \cite{Mordeson}has discovered the grand exploration of fuzzy semigroups, where theory of fuzzy semigroups is explored along with the applications of fuzzy semigroups in fuzzy coding, fuzzy finite state mechanics and fuzzy languages and the use of fuzzyfication in automata and formal language has widely been explored.

In this paper we present various characterizations of intra-regular $AG$-groupoids with left identity by the properties of their left ideals and bi-ideals. We also present characterizations by fuzzy left ideals and fuzzy bi-ideals of some special types.

\section{Preliminaries}

In this section we remind basic facts which will be need later. For simplicity a multiplication will be denoted by juxtaposition. Dots we will use only to avoid repetitions of brackets. For example, the formula $((xy)(zy))(xz)=(x(yy))z$ will be written as $(xy\cdot zy)\cdot xz=xy^2\cdot z$.

A groupoid $(S,\cdot)$ is called \textit{$AG$-groupoid}, if it satisfies the  \textit{left invertive law}: 
\begin{equation}  \label{e1}
ab\cdot c=cb\cdot a.
\end{equation}

Each $AG$-groupoid satisfies the \textit{medial law}: 
\begin{equation}  \label{e2}
ab\cdot cd=ac\cdot bd.
\end{equation}

Moreover, a \textit{unitary} $AG$-groupoid, i.e., an $AG$-groupoid with a left identity, satisfies the \textit{paramedial law}: 
\begin{equation}  \label{e3}
ab\cdot cd=db\cdot ca.
\end{equation}

In this case also holds 
\begin{equation}  \label{e4}
a\cdot bc=b\cdot ac ,
\end{equation}
\begin{equation}  \label{e5}
ab\cdot cd=dc\cdot ba .
\end{equation}
for all $a,b,c,d\in S$.

Let $S$ be an $AG$-groupoid. By \textit{$AG$-subgroupoid} of $S$ we mean a nonempty subset $A$ of $S$ such that $A^{2}\subseteq A$. By a \textit{left $( $right$)$ ideal} of $S$ we mean a nonempty subset $B$ of $S$ such that $SB\subseteq B$ (resp. $BS\subseteq B$). A \textit{two-sided ideal} or simply an \textit{ideal} of $S$ is a subset which is both a left and a right ideal of $S$. In a unitary $AG$-groupoid a right ideal is a left ideal, and consequently -- an ideal, but there are left ideals which are not a right ideal. Note also that the intersection of two left (right) ideals may be the empty set but the intersection of two ideals always is nonempty. Indeed, if $A,B$ are ideals of $S$ and $a\in A$, $b\in B$, then $ab\in A\cap B$. Obviously $A\cap B$ is an ideal of $S$. In a surjective $AG$-groupoid, i.e., in an $AG$-groupoid $S$ with the property $S=S^2$, each right ideal is a left ideal. The converse is not true.

By a \textit{generalized bi-ideal} (\textit{generalized interior ideal}) of $S$ we mean a nonempty subset $I$ of $S$ such that $(IS)I\subseteq I$ (resp. $(SI)S\subseteq I$). A generalized bi-ideal (generalized interior ideal) of $S$ which is an $AG$-subgroupoid is called a \textit{bi-ideal} (respectively, \textit{interior ideal}) of $S$. If $a^{2}\in A$ implies $a\in A$ for all $a\in S$, then we say that a subset $A\subseteq S$ is \textit{semiprime}.

A fuzzy subset $f$ of a set $S$ is described as an arbitrary function $f:S\longrightarrow \lbrack 0,1]$, where $[0,1]$ is the usual closed interval of real numbers. A fuzzy subset $f$ of $S$ of the form  
\begin{equation*}
f(z)=\left\{\begin{array}{lll}
t\in (0,1] & \mathrm{if } & z=x, \\ 
0 & \mathrm{if } & z\ne x
\end{array}
\right.
\end{equation*}
is called the \textit{fuzzy point} and is denoted by $x_t$. A fuzzy point $x_t$ is said to \textit{belong} to a fuzzy set $f,$ written as $x_t\in f,$ if $f(x)\geq t$, and \textit{quasi-coincident} with $f,$ written as $x_tq_kf, $ if $f(x)+t+k>1$, where $k$ is a fixed element of $[0,1)$. The symbol $x_t\in\vee q_kf$ means that $x_t\in f$ or $x_tq_kf$.

For any two fuzzy subsets $f$ and $g$ of $S$, $f\leq g$ means that, $f(a)\leq
g(a)$ for all $a\in S$. The symbols $f\wedge_k g$, $f\vee_k g$ and $f\circ_k g$, where $k\in [0,1)$,
denote the fuzzy subsets of $S$:

\medskip
$(f\wedge_k g)(a) =\min\{f(a),\,g(a),\,\frac{1-k}{2}\}=f(a)\wedge_k g(a)$,

\medskip
$(f\vee_k g)(a) =\max\{f(a),\,g(a),\,\frac{1-k}{2}\}=f(a)\vee_k g(a)$,

\medskip
$
\left(f\circ_k g\right)(a)=\left\{\!\!\!
\begin{array}
[c]{cl}
\bigvee \limits_{a=pq}\! \left \{f(p)\wedge_k g(q)\right\}& \text{ if
}\  \exists p,q\in S\text{ such that }a=pq,\\
0 & \text{ otherwise}%
\end{array}
\right.
$

\noindent
for all $a\in S$.

The set $S$ can be considered as a fuzzy subset of $S$ such that $S(a)=1$ for all $a\in S$. So, the fuzzy subset 
$(f\wedge_k S)(a)$ will be denoted as $f_k(a)$.

\medskip
A fuzzy subset $f$ of an $AG$-groupoid $S$ is called:

$-$ an \textit{$(\in ,\in\!\vee q_{k})$-fuzzy subgroupoid} if for all $x,y\in S$ and $r,t\in (0,1]$

$\ \ \ x_{r},y_{t}\in f\Rightarrow (xy)_{\min \{r,t\}}\in\!\vee q_{k}f$,

$-$ an \textit{$(\in ,\in\!\vee q_{k})$-fuzzy left ideal} if for all $x,y\in S$ and $t\in (0,1]$

$\ \ \ y_{t}\in f\Rightarrow (xy)_{t}\in\!\vee q_{k}f$,

$-$ an \textit{$(\in ,\in\!\vee q_{k})$-fuzzy right ideal} if for all $x,y\in S$ and $t\in (0,1]$

$\ \ \ x_{t}\in f\Rightarrow (xy)_{t}\in\!\vee q_{k}f$,

$-$ an \textit{$(\in ,\in\!\vee q_{k})$-fuzzy two-sided ideal} if it is both left and right $(\in ,\in\!\vee q_{k})$-fuzzy

$\ \ \ $ideal,

$-$ an \textit{$(\in ,\in\!\vee q_{k})$-fuzzy generalized bi-ideal} if for all $x,y,z\in S$ \ and \ $r,t\in (0,1]$

$\ \ \ x_{t},z_{r}\in f\Rightarrow (xy\cdot z)_{\min \{r,t\}}\in\!\vee q_{k}f $,

$-$ an \textit{$(\in ,\in\!\vee q_{k})$-fuzzy generalized interior ideal} if for all $x,y,z\in S$ and $t\in (0,1]$

$\ \ \ y_{t}\in f\Rightarrow (xy\cdot z)_{t}\in\!\vee q_{k}f$,

$-$ an \textit{$(\in ,\in\!\vee q_{k})$-fuzzy semiprime} if $f(a)\geq f_k(a^{2})$ \ for all $a\in S$.

\medskip

An $(\in ,\in\!\vee q_{k})$-fuzzy subgroupoid which is an $(\in ,\in\!\vee q_{k})$-fuzzy generalized bi-ideal (generalized interior ideal) is called an \textit{$(\in ,\in\!\vee q_{k})$-fuzzy bi-ideal} (respectively, an \textit{$(\in ,\in\!\vee q_{k})$-fuzzy interior ideal}).

\medskip
Similarly as in the case of semigroups (for details see \cite{Shabir2}) we can prove the following two propositions.

\begin{proposition}
\label{P2.2} A fuzzy subset $f$ of an $AG$-groupoid $S$ is

\begin{enumerate}
\item[$(i)$] an $(\in ,\in\!\vee q_{k})$-fuzzy subgroupoid if and only if $\,f(xy)\geq f(x)\wedge_k f(y)$,

\item[$(ii)$] an $(\in ,\in\!\vee q_{k})$-fuzzy left ideal if and only if $\,f(xy)\geq f_k(y)$,

\item[$(iii)$] an $(\in ,\in\!\vee q_{k})$-fuzzy right ideal if and only if $\,f(xy)\geq f_k(x)$,

\item[$(iv)$] an $(\in ,\in\!\vee q_{k})$-fuzzy bi-ideal if and only if $\,f(xy)\geq f(x)\wedge_k f(y)\,$ and

$f(xy\cdot z)\geq f(x)\wedge_k f(z)$,

\item[$(v)$] an $(\in ,\in\!\vee q_{k})$-fuzzy generalized bi-ideal if and only if $\,f(xy\cdot z)\geq f(x)\wedge_k f(z)$,

\item[$(vi)$] an $(\in ,\in\!\vee q_{k})$-fuzzy interior ideal if and only if $\,f(xy)\geq f(x)\wedge_k f(y)$ and $\,f(xy\cdot z)\geq f_k(y)$,

\item[$(vii)$] an $(\in ,\in\!\vee q_{k})$-fuzzy generalized interior ideal if and only if $\,f(xy\cdot z)\geq f_k(y)$.
\end{enumerate}

for all $x,y,z\in S$.
\end{proposition}

\begin{corollary}\label{C2.6}
In a unitary $AG$-groupoid each $(\in,\in\!\vee q_{k})$-fuzzy right ideal is an $(\in,\in\!\vee q_{k})$-fuzzy left ideal.
\end{corollary}
\begin{proof}
Indeed, $f(xy)=f(ee\cdot xy)=f(yx\cdot ee)\geq f_k(yx)\geq f_k(y) .$
\end{proof}

All such fuzzy subsets can be characterized by their levels, i.e., subsets of the form $U(f,t)=\{x\in S:f(x)\geq t\}$. Namely as a simple consequence of the transfer principle for fuzzy subsets (cf. \cite{kondo}) we obtain

\begin{proposition}
\label{P2.3} A fuzzy subset $f$ of an $AG$-groupoid of $S$ is its an $(\in,\in\!\vee q_{k})$-fuzzy subgroupoid $($left, right, interior ideal$)$ if and only if for all $0<t\leq\frac{1-k}{2}$ each nonempty level $U(f,t)$ is a subgroupoids $($left, right, interior ideal$)$ of $S$.
\end{proposition}

A similar result is valid for bi-ideals, generalized bi-ideals and generalized interior ideals.

\begin{definition}
\rm{Let $k$ be a fixed element of $[0,1)$. The characteristic function $\left(C_{A}\right)_{k}$ of a subset $A\subset S$ is defined as 
\begin{equation*}
\left( C_{A}\right) _{k}(x)=\left \{ 
\begin{array}{lcl}
t\geq \frac{1-k}{2} & \text{ if } & x\in A, \\ 
0 & \text{ if } & x\notin A.
\end{array}
\right.
\end{equation*}
}
\end{definition}

\medskip

The proof of the following proposition is very similar to the proof of analogous results for semigroups (cf. \cite{Shabir2}).

\begin{proposition}
\label{P2.5} Let $J$ be a nonempty subset of an $AG$-groupoid $S$. Then:

\begin{enumerate}
\item[$(i)$] $J$ is an ideal of $S$ if and only if $\left( C_{J}\right)_{k}$ is an $(\in ,\in\!\vee q_{k})$-fuzzy ideal of $S$,

\item[$(ii)$] $J$ is a left $($right$)$ ideal of $S$ if and only if $\left(C_{J}\right) _{k}$ is an $(\in ,\in\!\vee q_{k})$-fuzzy left $($right$)$ ideal of $S$,

\item[$(iii)$] $J$ is a bi-ideal of $S$ if and only if $\left(C_{J}\right)_{k}$ is an $(\in ,\in\!\vee q_{k})$-fuzzy bi-ideal of $S$,

\item[$(iv)$] $J$ is an interior ideal if and only if $\left( C_{J}\right)_{k}$ is an $(\in ,\in\!\vee q_{k})$-fuzzy interior ideal of $S$,

\item[$(v)$] $J$ is semiprime if and only if $\left( C_{J}\right) _{k}$ is an $(\in ,\in\!\vee q_{k})$-fuzzy semiprime.
\end{enumerate}
\end{proposition}

The following lemma is obvious, so we omit the proof.

\begin{lemma}
\label{L2.7} Let $A,B$ be nonempty subsets of an $AG$-groupoid $S$. Then:

\begin{enumerate}
\item[$(i)$] \ $\left( C_{A\cap B}\right) _{k}=\left(C_{A}\wedge_{k}C_{B}\right)$,

\item[$(ii)$] \ $\left( C_{A\cup B}\right) _{k}=\left(C_{A}\vee_{k}C_{B}\right)$,

\item[$(iii)$] \ $\left( C_{AB}\right) _{k}=\left(C_{A}\circ_{k}C_{B}\right) .$
\end{enumerate}
\end{lemma}

\section{Ideals of intra-regular AG-groupoids}

\begin{definition}
\rm{An element $a$ of an $AG$-groupoid $S$ is called \textit{intra-regular} if there exist $x,y\in S$ such that $a=xa^{2}\cdot y$. If all elements of $S$ are intra-regular, then we say that an $AG$-groupoid $S$ is \textit{intra-regular}. }
\end{definition}

\begin{example}\rm
Let $(G,\circ,e)$ be an arbitrary abelian group. Then, as it is not difficult to see, $G$ with the operation $xy=x^{-1}\circ y$ is an intra-regular $AG$-groupoid with the left identity $e$. This groupoid is not a semigroup. It is a special case of transitive distributive Steiner quasigroups (cf. \cite{Dud}). Moreover, in this $AG$-groupoid a fuzzy subset $f$ of $G$ is an $(\in,\in\!\vee q_{k})$-fuzzy subgroupoid of $(G,\cdot)$ if and only if it is an $(\in,\in\!\vee q_{k})$-fuzzy subgroup of the group $(G,\circ,e)$. 
\end{example}

\begin{example}\rm
It is not difficult to verify that the set $S=\left \{1,2,3,4,5,6\right \}$ with the multiplication defined by the table 
\begin{equation*}
\begin{tabular}{l|llllll}
$\cdot $ & $1$ & $2$ & $3$ & $4$ & $5$ & $6$ \\ \hline
$1$ & $1$ & $1$ & $1$ & $1$ & $1$ & $1$ \\ 
$2$ & $1$ & $2$ & $1$ & $1$ & $1$ & $1$ \\ 
$3$ & $1$ & $1$ & $5$ & $6$ & $3$ & $4$ \\ 
$4$ & $1$ & $1$ & $4$ & $5$ & $6$ & $3$ \\ 
$5$ & $1$ & $1$ & $3$ & $4$ & $5$ & $6$ \\ 
$6$ & $1$ & $1$ & $6$ & $3$ & $4$ & $5$
\end{tabular}
\end{equation*}
is an $AG$-groupoid. It is intra-regular because $1=(1\cdot 1^{2})\cdot 1$, $2=(2\cdot 2^{2})\cdot 2$, $3=(3\cdot 3^{2})\cdot 5$, $4=(6\cdot 4^{2})\cdot 3 $, $5=(5\cdot 5^{2})\cdot 5$ and $6=(4\cdot 6^{2})\cdot 3$. Moreover, $A=\{1\}$ and $B=\{1,2\}$ are its ideals. A fuzzy subset $f$ of $S$ such that $f(1)=0.9$, $f(2)=0.8$ and $f(x)=0.5$ otherwise, is an $(\in ,\in\!\vee
q_{k})$-fuzzy ideal of $S$.
\end{example}

\begin{lemma}
\label{L3.4} In a unitary intra-regular $AG$-groupoid $G$ for every $a\in G$ there exist $x,y,z\in G$ such that

\begin{enumerate}
\item[$(i)$] \ $a=a\cdot za$,

\item[$(ii)$] \ $a=wa\cdot a$,

\item[$(iii)$] \ $a=(a^2\cdot x^2y^2)a$,

\item[$(iv)$] \ $a=(a\cdot x^2y^2)a^2$,

\item[$(v)$] \ $a=a^2\cdot (ay^2\cdot x^2)$,

\item[$(vi)$] \ $a=a(y^2x^2\cdot a)\cdot a$,

\item[$(vii)$] \ $a=a^2z\cdot a^2$,

\item[$(viii)$] \ $a^2=az\cdot a$.
\end{enumerate}
\end{lemma}
\begin{proof}
$(i).$ \ Let $S$ be a unitary intra-regular $AG$-groupoid with the left identity $e$. Since for every $a\in S$ there exist $x,y\in S$ such that $a=xa^{2}\cdot y$, we have 
$$
a =xa^{2}\cdot y=xa^{2}\cdot ey\overset{\eqref{e5}}{=}ye\cdot a^{2}x\overset{\eqref{e4}}{=}a^{2}(ye\cdot x)\\
\overset{\eqref{e5}}{=}(x\cdot ye)a^{2}=z\cdot a^2\overset{\eqref{e4}}{=}a\cdot za
$$
for $z=x\cdot ye$.

$(ii).$ From $(i)$ we obtain
$$
a=a\cdot za=ea\cdot za\overset{\eqref{e2}}{=}ez\cdot a^2 \overset{\eqref{e5}}{=}a^2\cdot ze\overset{\eqref{e1}}{=}(ze\cdot a)a=wa\cdot a
$$
wchich proves $(ii)$. 

$(iii)$. Using \eqref{e4} and \eqref{e1}, we have 
\arraycolsep=.5mm 
\begin{eqnarray*}
a&=&xa^{2}\cdot y=(x\cdot aa)y=(a\cdot xa)y \overset{\eqref{e1}}{=}(y\cdot xa)a=(y\cdot x(xa^{2}\cdot y))a \\
&=&(y\cdot (xa^{2}\cdot xy))a=(xa^{2}\cdot (y\cdot xy))a=(xa^{2}\cdot xy^2)a=(a^{2}\cdot x^{2}y^{2})a.
\end{eqnarray*}
This proves $(iii)$. Applying \eqref{e1} to $(iii)$ we obtain $(iv)$. Now, $(iv)$ together with \eqref{e2} and \eqref{e1} imply $(v)$. $(vi)$ is a consequence of $(iii)$. Indeed, \arraycolsep=.5mm 
\begin{eqnarray*}
a&=&(a^{2}\cdot x^{2}y^{2})a\overset{\eqref{e2}}{=}(x^{2}\cdot a^{2}y^{2})a
\overset{\eqref{e4}}{=}(a^2\cdot x^{2}y^{2})a\overset{\eqref{e3}}{=}(y^{2}a\cdot x^{2}a)a \\
&\overset{\eqref{e2}}{=}&(y^{2}x^{2}\cdot aa)a\overset{\eqref{e4}}{=}(a\cdot (y^{2}x^{2}\cdot a))a=a(y^2x^2\cdot a)\cdot a.
\end{eqnarray*}
Similarly, using \eqref{e4} and the left identity $e$ of $G$, we have 
\begin{eqnarray*}
a &=&xa^{2}\cdot y=(a\cdot xa)y=(a\cdot x(xa^{2}\cdot y))y=(a\cdot (xa^{2}\cdot xy))y \\
&=&(xa^{2}\cdot (a\cdot xy))y\overset{\eqref{e1}}{=}(y(a\cdot xy))\cdot xa^{2}=(y(ea\cdot xy))\cdot xa^{2} \\
&\overset{\eqref{e3}}{=}&(y(ya\cdot xe))\cdot xa^{2}= (ya\cdot (y\cdot xe))\cdot xa^{2},
\end{eqnarray*}
i.e., $\,a=(ya\cdot u)\cdot xa^{2}$, where $u=y\cdot xe$.

Further 
\begin{eqnarray*}
a&=&(ya\cdot u)\cdot xa^{2}=(y(xa^{2}\cdot y)\cdot u)\cdot xa^{2}=((xa^{2}\cdot y^2)\cdot u)\cdot xa^{2} \\
&\overset{\eqref{e1}}{=}&(uy^{2}\cdot xa^{2})\cdot xa^{2}\overset{\eqref{e1}}{=}(xa^{2}\cdot xa^{2})\cdot uy^{2} \overset{\eqref{e2}}{=}(x^2\cdot a^{2}a^{2})\cdot uy^{2} \\
&\overset{\eqref{e3}}{=}&(a^{2}x\cdot a^{2}x)\cdot uy^{2} \overset{\eqref{e1}}{=}(a^{2}x\cdot x)a^{2}\cdot uy^{2}\overset{\eqref{e1}}{=}(x^2\cdot a^{2})a^{2}\cdot uy^{2} \\
&=&(x^2a^{2}\cdot a^{2})\cdot uy^{2} \overset{\eqref{e1}}{=}(uy^{2}\cdot a^{2})\cdot x^{2}a^{2} \overset{\eqref{e2}}{=}(uy^{2}\cdot x^{2})\cdot a^{2}a^{2} \\
&=&a^{2}\cdot (uy^{2}\cdot x^{2})a^{2} \overset{\eqref{e2}}{=}a(uy^{2}\cdot x^{2})\cdot aa^{2} \overset{\eqref{e3}}{=}a^{2}(uy^{2}\cdot x^{2})\cdot a^2.
\end{eqnarray*}
Thus $\,a=a^2z\cdot a^2$ for $z=uy^2\cdot x^2$. This proves $(vii)$.

To prove $(viii)$ observe first that 
\begin{eqnarray*}
a^2&=&a(xa^2\cdot y)\overset{\eqref{e4}}{=}xa^2\cdot ay\overset{\eqref{e5}}{=}ya\cdot a^2x\overset{\eqref{e4}}{=}a^2\cdot (ya\cdot x)\overset{\eqref{e1}}{=}(ya\cdot x)a\cdot a,
\end{eqnarray*}
i.e., 
\begin{equation}  \label{e6}
a^2=ua\cdot a
\end{equation}
for $u=ya\cdot x$.

On the other hand 
\begin{eqnarray*}
ua&=&u(xa^2\cdot y)\overset{\eqref{e4}}{=}xa^2\cdot uy\overset{\eqref{e5}}{=}%
yu\cdot a^2x\overset{\eqref{e4}}{=}a^2\cdot (yu\cdot x) \\
&\overset{\eqref{e5}}{=}&(x\cdot yu)\cdot a^2 \overset{\eqref{e4}}{=}a\cdot
(x\cdot yu)a= az,
\end{eqnarray*}
where $\,z=(x\cdot yu)a$.

This together with \eqref{e6} proves $(vi)$.
\end{proof}

As a simple consequence of Lemma \ref{L3.4} we obtain

\begin{corollary}
\label{C3.5} In unitary intra-regular $AG$-groupoids fuzzy left $($right$)$ ideals and fuzzy generalized bi-ideals $($interior ideals$)$ are semiprime.
\end{corollary}

It is not difficult to see that generalized interior ideals are semiprime also in intra-regular $AG$-groupoids which are not unitary.

\begin{corollary}
\label{C3.6} In an intra-regular $AG$-groupoid with a left identity $e$ we have $f(a)=f(a^2)\geq f(e)$ for all fuzzy left $($right$)$ ideals, fuzzy generalized bi-ideals and fuzzy generalized interior ideals.
\end{corollary}
\begin{proof}
For fuzzy left (right) ideals it is clear. Using Lemma \ref{L3.4} $(vii)$ and $(viii)$ we obtain $f(a)=f(a^2)$ for fuzzy generalized bi-ideals. In this case also $f(a^2)=f(ea\cdot a)=f(a^2\cdot e)=f(ea^2\cdot e)\geq f(e)$.

For fuzzy generalized interior ideals we have $f(a^2)=f(ea\cdot a)\geq f(a)=f(xa^2\cdot y)\geq f(a^2)$. So $f(a)=f(a^2)$ for every $a\in G$. Moreover, $f(a^2)=f(e^2\cdot a^2)\geq f(e)$.
\end{proof}

Comparing the above results with Proposition \ref{P2.2} we can see that the corresponding results are valid for $(\in ,\in\!\vee q_{k})$-fuzzy ideals too. Moreover, as a consequence of Proposition \ref{P2.5} we obtain

\begin{corollary}
\label{C3.7} In unitary intra-regular $AG$-groupoids all left $($right$)$ ideals and all generalized bi-ideals $($interior ideals$)$ are semiprime.
\end{corollary}

\begin{theorem}
\label{T3.8} A unitary $AG$-groupoid $S$ is intra-regular if and only if $a\in Sa^2$ for every $a\in S$.
\end{theorem}
\begin{proof}
Let $S$ be an $AG$-groupoid and let $e$ be its left identity. Then for every $a\in S$ we have $Sa^2\cdot S\overset{\eqref{e4}}{=}(a\cdot Sa)S\overset{\eqref{e1}}{=}(S\cdot Sa)a=(eS\cdot Sa)a\overset{\eqref{e3}}{=}(aS\cdot Se)a\subseteq (aS\cdot S)a\overset{\eqref{e1}}{=}aS\cdot aS\overset{\eqref{e3}}{=}S^2a^2=Sa^2$, which shows that $Sa^2\cdot S\subseteq Sa^2$. Obviously, $a^2\in Sa^2$.

If $S$ is intra-regular, then for every $a\in S$ we obtain 
\begin{equation*}
a\in Sa^2\cdot S\subseteq (S\cdot Sa^2)\cdot S=(S\cdot Sa^2)\cdot eS=Se\cdot
(Sa^2\cdot S)\subseteq Se\cdot a^2=S^2\cdot ea^2=Sa^2.
\end{equation*}
So, $a\in Sa^2$.

Conversely, since $a\in Sa^2$ for every $a\in S$, thus  
\begin{equation*}
a\in Sa^2=eS\cdot a^2=a^2S\cdot e\subseteq a^2S\cdot S=(a^2\cdot eS)\cdot S%
\overset{\eqref{e5}}{=}(Se\cdot a^2)\cdot S\subseteq Sa^2\cdot S .
\end{equation*}
Hence $a\in Sa^2\cdot S$.
\end{proof}

\begin{corollary}
\label{C3.9} A unitary $AG$-groupoid is intra-regular if and only if all its right ideals $($or equivalently: all its interior ideals$\,)$ are semiprime.
\end{corollary}
\begin{proof}
By Corollary \ref{C3.7} in any unitary intra-regular $AG$-groupoid right ideals and interior ideals are semiprime. On the
other side, from the first part of the proof of Theorem \ref{T3.8} it follows that $Sa^2$ is a right ideal of each unitary $AG$-groupoid $S$. Moreover, $Sa^2$ is also an interior ideal. Indeed, 
\begin{equation*}
(S\cdot Sa^2)\cdot S=(S\cdot Sa^2)\cdot eS\overset{\eqref{e5}}{=}Se\cdot
(Sa^2\cdot S)\subseteq Se\cdot Sa^2\overset{\eqref{e2}}{=}Sa^2 .
\end{equation*}

So, if all right ideals or all interior ideals of $S$ are semiprime, then, by Theorem \ref{T3.8}, a unitary $AG$-groupoid $S$ is intra-regular.
\end{proof}

Using these results, we can get useful characterizations of unitary intra-regular $AG $-groupoids by their ideals of various types. Let's start with the characterizations by left ideals.

\begin{theorem}
\label{T3.10} For a unitary $AG$-groupoid $S$ the following conditions are equivalent.

\begin{enumerate}
\item[$(i)$] $S$ is intra-regular.
\item[$(ii)$] $A\cap B\cap C\subseteq (AB)C$ for all subsets of $S$ when one of them is a left ideal,
\item[$(iii)$] $f\wedge_{k}g\wedge_{k}h\leq (f\circ_{k}g)\circ_{k}h$ \ for all $(\in,\in\!\vee q_{k})$-fuzzy subsets of $S$ when one of them is an $(\in,\in\!\vee q_{k})$-fuzzy left ideal.
\end{enumerate}
\end{theorem}
\begin{proof}
$(i)\Rightarrow (iii)$ \ Assume that $f$ is an $(\in,\in\!\vee q_{k})$-fuzzy left ideal of $S$. Since by Lemma \ref{L3.4} $(iii)$
and \eqref{e1} for every $a\in S$ there are $x,y\in S$ such that $a=(a^2\cdot x^2y^2)a=((x^2y^2\cdot a)a)a$ we have \arraycolsep=.5mm 
\begin{eqnarray*}
((f\circ _{k}g)\circ _{k}h)(a)&=&\bigvee \limits_{a=pq}\left \{ (f\circ _{k}g)(p) \wedge_k h(q)\right\} \\
&=&\bigvee \limits_{a=pq}\left\{ \left( \bigvee \limits_{p=uv}f(u) \wedge_k g(v)\right) \wedge_k h(q)\right\} \\
&=&\bigvee \limits_{a=uv\cdot q}\left\{\left(f(u)\wedge_k g(v)\right)\wedge_k h(q)\right\} \\
&=&\bigvee \limits_{a=((x^{2}y^{2}\cdot a)a)a=uv\cdot q}\!\!\!\!\!\left \{ (f(u)\wedge_k g(v))\wedge_k h(q)\right \} \\
&\geq &\left \{ f(x^{2}y^{2}\cdot a) \wedge_k g(a) \right \} \wedge_k h(a) \\
&\geq &\left(f_k(a)\wedge_k g(a)\right) \wedge_k h(a)\\
&=&((f\wedge _{k} g)\wedge _{k} h)(a)
\end{eqnarray*}
for $(\in,\in\!\vee q_{k})$-fuzzy subsets $g,h$ of $S$. Thus\ $(f\wedge_{k}g)\wedge_{k}h\leq (f\circ_{k}g)\circ_{k}h.$

In the case when $g$ is an $(\in,\in\!\vee q_{k})$-fuzzy left ideal of $S$ the proof is similar but we must use the equation $(iv)$ from Lemma \ref{L3.4}.

In the last case when $h$ is an $(\in,\in\!\vee q_{k})$-fuzzy left ideal we must use the equation $a=a^2\cdot (x^2y^2\cdot a)$ which is a consequence of Lemma \ref{L3.4} $(iv)$ and \eqref{e3}.

\smallskip

$(iii)\Rightarrow (ii)$ \ Assume that $A$ is a left ideal of $S$ and $B,C$ are arbitrary subsets of $S$. Then by Proposition \ref{P2.5}, $(C_{A})_{k}$\ is an $(\in,\in\!\vee q_{k})$-fuzzy left ideal of $S$ and $(C_{B})_{k},(C_{C})_{k}$\ are $(\in,\in\!\vee q_{k})$-fuzzy subsets of $S$. Thus, by Lemma \ref{L2.7} and $(iii)$ we have 
\begin{equation*}
(C_{(A\cap B)\cap C})_{k}=(C_A\wedge_{k}C_B)\wedge_{k} C_C\leq (C_A\circ_{k}C_B)\circ_{k}C_C=(C_{(AB)C})_{k}.
\end{equation*}
Therefore\ $A\cap B\cap C\subseteq (AB)C$.

\smallskip $(ii)\Rightarrow (i)$ \ Since $S$ has a left identity, for every $a\in S$ we have 
\begin{equation*}
S\cdot Sa=eS\cdot Sa\overset{\eqref{e3}}{=}aS\cdot Se\subseteq aS\cdot S\overset{\eqref{e1}}{=}S^2\cdot a=Sa .
\end{equation*}
So, $Sa$ is a left ideal of $S$. Obviously, $a\in Sa$. Thus 
\begin{equation*}
a\in Sa\cap Sa\cap Sa\subseteq (Sa\cdot Sa)\cdot Sa\overset{\eqref{e2}}{=}(S^2\cdot a^2)\cdot Sa\subseteq Sa^{2}\cdot S ,
\end{equation*}
which shows that $S$ is an intra-regular $AG$-groupoid.
\end{proof}

\begin{corollary}
\label{C3.11} For a unitary $AG$-groupoid $S$ the following conditions are equivalent.

\begin{enumerate}
\item[$(i)$] $S$ is intra-regular.
\item[$(ii)$] $A\cap B\cap C\subseteq (AB)C$ for all left ideals of $S$,
\item[$(iii)$] $f\wedge_{k}g\wedge_{k}h\leq (f\circ_{k}g)\circ_{k}h$ \ for
all $(\in,\in\!\vee q_{k})$-fuzzy left ideals of $S$.
\end{enumerate}
\end{corollary}

\begin{theorem}
\label{T3.12} For a unitary $AG$-groupoid $S$ the following are equivalent.

\begin{enumerate}
\item[$(i)$] $S$ is intra-regular.
\item[$(ii)$] $A\cap B\subseteq AB\cap BA$ \ for all left ideals of $S$,
\item[$(iii)$] $f\wedge_{k}g\leq (f\circ_{k}g)\wedge (g\circ_{k}\!f)$ \ for
all $(\in ,\in\!\vee q_{k})$-fuzzy left ideals of $S$.
\end{enumerate}
\end{theorem}
\begin{proof}
$(i)\Rightarrow (iii)$ \ Since $S$ is a unitary intra-regular $AG$-groupoid, by Lemma \ref{L3.4} $(vi)$,
for every $a\in S$ there exist $x,y\in S$ such that $a=a(y^2x^2\cdot a)\cdot a$. Therefore, for all $(\in ,\in\!\vee q_{k})$-fuzzy left ideals of $S$ we have \arraycolsep=.5mm 
\begin{eqnarray*}
(f\circ_{k}g)(a) &=&\bigvee \limits_{a=pq}\!\left\{f(p)\wedge_k g(q)\right\} 
=\bigvee \limits_{a=a(y^2x^2\cdot a)\cdot a=pq}\!\!\!\left\{f(p) \wedge_k g(q)\right\} \\
&\geq &f(a(y^2x^2\cdot a))\wedge_k g(a)\\
&\geq&f_k(a)\wedge_k g(a)=(f\wedge _{k}g)(a) .
\end{eqnarray*}%
Thus $f\wedge_{k}g\leq f\circ_{k}g$. Similarly we can show that $f\wedge_{k}g\leq g\circ_{k}\!f$. Consequently, $f\wedge_{k}g\leq
(f\circ_{k}g)\wedge (g\circ_{k}\!f)$.

\smallskip $(iii) \Rightarrow (ii)$ \ Analogously as in the proof of Theorem \ref{T3.10}.

\smallskip $(ii)\Rightarrow (i)$ \ Since $Sa$ is a left ideal of $S$, for every $a\in S$ we have 
\begin{eqnarray*}
a &\in &Sa\cap Sa\subseteq (Sa\cdot Sa)\cap (Sa\cdot Sa)= Sa\cdot Sa\overset{\eqref{e3}}{=}a^2\cdot S^2 \\
&\overset{\eqref{e4}}{=}&S\cdot a^2S=S^2\cdot a^2S\overset{\eqref{e2}}{=}Sa^{2}\cdot S^2=Sa^{2}\cdot S ,
\end{eqnarray*}
which shows that $S$ is an intra-regular $AG$-groupoid.
\end{proof}

\begin{theorem}
\label{T3.13} For a unitary $AG$-groupoid $S$ the following are equivalent.

\begin{enumerate}
\item[$(i)$] $S$ is intra-regular,
\item[$(ii)$] $A\cap B\cap C=(AB)C$ \ for each bi-ideal $A$ and arbitrary subsets $B,C$ of $S$,
\item[$(iii)$] $f\wedge _{k}g\wedge _{k}h\leq (f\circ_{k}g)\circ_{k}\!h$ \
for each $(\in,\in\!\vee q_{k})$-fuzzy bi-ideal $f$ and all $(\in,\in\!\vee
q_{k})$-fuzzy subsets $g,h$ of $S$,
\item[$(iv)$] $f\wedge _{k}g\wedge _{k}h\leq (f\circ _{k}g)\circ _{k}\!h$ \ for each 
$(\in ,\in\!\vee q_{k})$-fuzzy generalized bi-ideal $f$ and all $(\in,\in\!\vee q_{k})$-fuzzy subsets $g,h$ of $S$.
\end{enumerate}
\end{theorem}
\begin{proof}
$(i)\Rightarrow (iv)$ \ By Lemma \ref{L3.4} $(vi)$ for every $a\in S$ there exist $x,y\in S$ such that $a=a(y^2x^2\cdot a)\cdot a$. Hence, using \eqref{e4}, we obtain \arraycolsep=.5mm 
\begin{eqnarray*}
a &=&a(y^{2}x^{2}\cdot a)\cdot a=a(y^{2}x^{2}\cdot (xa^{2}\cdot y))\cdot a 
\overset{\eqref{e4}}{=}a(xa^{2}\cdot (y^{2}x^{2}\cdot y))\cdot a \\
&\overset{\eqref{e2}}{=}&a((x\cdot y^{2}x^{2})\cdot a^{2}y)\cdot a \overset{\eqref{e4}}{=}a(a^{2}\cdot (x\cdot y^{2}x^{2})y)\cdot a \\
&\overset{\eqref{e4}}{=}&a^{2}(a\cdot (x\cdot y^{2}x^{2})y)\cdot a \overset{%
\eqref{e1}}{=}((a\cdot(x\cdot y^{2}x^{2})y)a\cdot a)\cdot a.
\end{eqnarray*}
Thus, for an arbitrary $(\in ,\in\!\vee q_{k})$-fuzzy generalized bi-ideal $f$ and all $(\in ,\in\!\vee q_{k})$-fuzzy subsets $g,h$ of $S$ we have 
\begin{eqnarray*}
((f\circ _{k}g)\circ _{k}\!h)(a) &=&\bigvee\limits_{a=pq}\left\{(f\circ_{k}g)(p)\wedge_k h(q)\right\}\\
&=&\bigvee \limits_{a=pq}\left\{\bigvee\limits_{p=uv}\left( f(u)\wedge_k g(v)\right)\wedge_k h(q)\right\} \\
&=&\bigvee \limits_{a=uv\cdot q}\left\{f(u)\wedge_k g(v)\wedge_k h(q)\right\} \\
&=&\bigvee \limits_{a=([a\cdot (x\cdot y^{2}x^{2})y]a\cdot a)\cdot a=uv\cdot
q}\!\!\!\!\!\!\!\!\!\! \left\{f(u)\wedge_k g(v)\wedge_k h(q)\right\} \\
&\geq &f([\,a\cdot (x\cdot y^{2}x^{2})y]a)\wedge_k g(a)\wedge_k h(a) \\
&\geq &f_k(a)\wedge_k g(a)\wedge_k h(a) =(f\wedge _{k}g\wedge_{k}h)(a) .
\end{eqnarray*}
So, $f\wedge_{k}g\wedge_{k}h\leq (f\circ_{k}g)\circ_{k}\!h$.

\smallskip

$(iv)\Rightarrow (iii)$ \ Obvious.

\smallskip

$(iii)\Rightarrow (ii)$ \ Similarly as in the proof of Theorem \ref{T3.10}.

\smallskip

$(ii)\Rightarrow (i)$ \ Since $S$ has a left identity, so we have 
$$(Sa\cdot S)Sa\subseteq S^{2}\cdot Sa\overset{\eqref{e3}}{=}aS\cdot S^{2}\subseteq aS\cdot S\overset{\eqref{e1}}{=}S^{2}\cdot a\subseteq Sa$$ for every $a\in S$. This means that $Sa$ is a bi-ideal of $S$. Thus, by $(ii)$ and \eqref{e2}, we obtain $a\in Sa\cap Sa\subseteq (Sa\cdot Sa)\cdot Sa=(S^{2}\cdot a^{2})\cdot Sa\subseteq Sa^{2}\cdot S$. Hence $S$ intra-regular.
\end{proof}

\begin{corollary}
\label{C3.14} For a unitary $AG$-groupoid $S$ the following are equivalent.

\begin{enumerate}
\item[$(i)$] $S$ is intra-regular,
\item[$(ii)$] $A\cap B\cap C=(AB)C$ \ for all bi-ideals of $S$,
\item[$(iii)$] $f\wedge _{k}g\wedge _{k}h\leq (f\circ_{k}g)\circ_{k}\!h$ \
for all $(\in ,\in\!\vee q_{k})$-fuzzy bi-ideals of $S,$\ 
\item[$(iv)$] $f\wedge _{k}g\wedge _{k}h\leq (f\circ _{k}g)\circ _{k}\!h$ \
for all $(\in ,\in \!\vee q_{k})$-fuzzy generalized bi-ideals of $S$.
\end{enumerate}
\end{corollary}
\begin{theorem}
\label{T3.15} For a unitary $AG$-groupoid $S$ the following are equivalent.

\begin{enumerate}
\item[$(i)$] $S$ is intra-regular.
\item[$(ii)$] $A\cap B\subseteq AB\cap BA$ \ for any bi-ideal $A$ and any subset $B$ of $S$,
\item[$(iii)$] $f\wedge _{k}g\leq (f\circ _{k}g)\wedge (g\circ _{k}\!f)$ \ for any 
$(\in ,\in \!\vee q_{k})$-fuzzy bi-ideal $f$ and any $(\in ,\in\!\vee q_{k})$-fuzzy subset $g$ of $S$.
\end{enumerate}
\end{theorem}
\begin{proof}
$(i)\Rightarrow (iii)$ \ Since $S$ is a unitary intra-regular $AG$-groupoid and for $z\in S$ there exist $u,v$ in $S$ such that $z=uv$. Then from Lemma \ref{L3.4} $(vi)$
\begin{equation*}
a=a^{2}z\cdot a^{2}\overset{\eqref{e5}}{=}a^{2}\cdot za^{2}\overset{\eqref{e1}}{=}(za^{2}\cdot a)a\overset{\eqref{e5}}{=}(a^{2}z^{\prime }\cdot a)a ,
\end{equation*}
where $z^{\prime }=vu$. Therefore, for an arbitrary $(\in ,\in \!\vee q_{k})$-fuzzy bi-ideal $f$ of $S$ and an arbitrary $(\in ,\in \!\vee q_{k})$-fuzzy subset $g$ of $S$ we have \arraycolsep=.5mm 
\begin{eqnarray*}
(f\circ _{k}g)(a) &=&\bigvee\limits_{a=pq}\!\left\{ f(p)\wedge_k g(q)\right\} 
=\bigvee\limits_{a=(a^{2}z^{\prime }\cdot a)a=pq}\!\!\!\!\!\!\left\{ f(p)\wedge_k g(q)\right\} \\
&\geq &f(a^{2}z^{\prime }\cdot a)\wedge_k g(a) \\
&\geq &f_k(a^{2})\wedge_k g(a) =(f\wedge _{k}g)(a).
\end{eqnarray*}
Thus $f\wedge _{k}g\leq f\circ _{k}g$. Similarly we can show that $f\wedge _{k}g\leq g\circ _{k}\!f$. Consequently, $f\wedge _{k}g\leq (f\circ _{k}g)\wedge (g\circ _{k}\!f)$.

$(iii)\Rightarrow (ii)$ \ Analogously as in the proof of Theorem \ref{T3.10}.

$(ii)\Rightarrow (i)$ \ Analogously as in the proof of Theorem \ref{T3.12}.
\end{proof}

\begin{corollary}
For a unitary $AG$-groupoid $S$ the following are equivalent.
\end{corollary}

\begin{enumerate}
\item[$(i)$] $S$ is intra-regular.
\item[$(ii)$] $A\cap B\subseteq AB\cap BA$ \ for all bi-ideals of $S$,
\item[$(iii)$] $f\wedge _{k}g\leq (f\circ _{k}g)\wedge (g\circ _{k}\!f)$ \
for all $(\in ,\in \!\vee q_{k})$-fuzzy bi-ideals of $S$.
\end{enumerate}

\medskip

Using the same method as in the proofs of Theorems \ref{T3.10} and \ref{T3.12} we can obtain the following characterization of unitary intra-regular $AG$-groupoids by their left ideals and bi-ideals.

\begin{theorem}
\label{T3.17} A unitary $AG$-groupoid $S$ is intra-regular if and only if one of the following conditions is satisfied:

\begin{enumerate}
\item[$(i)$] $A=A^2$ for all left ideals of $S$,
\item[$(ii)$] $A=A^2A$ for all left ideals of $S$,
\item[$(iii)$] $f_{k}\leq f\circ_{k}f$ \ for all $(\in,\in\!\vee q_{k})$-fuzzy left ideals of $S$,
\item[$(iv)$] $f_{k}\leq (f\circ_{k}f)\circ_{k}f$ \ for all $(\in,\in\!\vee q_{k})$-fuzzy left ideals of $S$.
\end{enumerate}
\end{theorem}

The above theorem is also valid if we replace the equation $A=A^2$ by the inclusion $A\subseteq A^2$ and left ideals by bi-ideals or generalized bi-ideals.

\section{Quasi-ideals of intra-regular AG-groupoids}

By a \textit{quasi-ideal} of an $AG$-groupoid $S$ we mean a nonempty subset $Q$ of $S$ such that $SQ\cap QS\subseteq Q$. Obviously each left (right) ideal is a quasi-ideal.

\begin{proposition}\label{P4.1}
If $S$ is a unitary $AG$-groupoid, then $Sa\cap aS$, $Sa$ and $Sa^2$ are quasi-ideals for every $a\in S$.
\end{proposition} 
\begin{proof} Indeed, using \eqref{e1} and \eqref{e4}, for every $a\in S$ we get
\[
\arraycolsep=.5mm
\begin{array}{rll}
S(Sa\cap aS)\cap(Sa\cap aS)S&=&S(Sa)\cap S(aS)\cap(Sa)S\cap (aS)S\\
&=&S(Sa)\cap aS\cap(Sa)S\cap Sa\\
&\subseteq&Sa\cap aS , 
\end{array}
\]
which shows that $Sa\cap aS$ is a quasi-ideal. Moreover,
$$
S(Sa)\cap(Sa)S\subseteq S(Sa)=SS\cdot Sa\overset{\eqref{e5}}{=}aS\cdot
SS=aS\cdot S\overset{\eqref{e1}}{=}SS\cdot a=Sa 
$$
and
$$
S(Sa^2)\cap(Sa^2)S\subseteq (Sa^2)S=Sa^{2}\cdot S^2\overset{\eqref{e5}}{=}S^2\cdot a^2S=
S\cdot a^2S\overset{\eqref{e4}}{=}a^2\cdot SS\overset{\eqref{e5}}{=}SS\cdot a^{2}=Sa^2.
$$
Hence $Sa$ and $Sa^2$ are quasi-ideals of $S$.
\end{proof}

From the above proof we obtain 

\begin{corollary}\label{C4.2}
In a unitary $AG$-groupoid $S$ we have
$
Sa\cdot Sa=Sa^2=Sa^2\cdot S
$
for every $a\in S$.
\end{corollary}

\begin{proposition}\label{P4.3}
Quasi-ideals of a unitary $AG$-groupoid are semiprime.
\end{proposition} 
\begin{proof} In fact, if $Q$ is a quasi-ideal of $S$ and $a^2\in Q$, then by Lemma \ref{L3.4} $(vii)$ and $(v)$ we see that $a\in SQ\cap QS\subseteq Q$.
\end{proof}

\begin{definition}\rm 
A fuzzy subset $f$ of an $AG$-groupoid $S$ is called an \textit{$(\in,\in \! \vee q_{k})$-fuzzy} \textit{quasi-ideal} of $S$ if $f\geq (S\circ f)\wedge_k (f\circ S)$.
\end{definition}
 
As a simple consequence of the transfer principle for fuzzy subsets (cf. \cite{kondo}) we have
\begin{proposition}\label{P4.5}
A fuzzy subset $f$ of an $AG$-groupoid $S$ is its $(\in,\in \! \vee q_{k})$-fuzzy quasi-ideal if and only if for all $0<t\leq\frac{1-k}{2}$ each nonempty level $U(f,t)$ is a quasi-ideal of $S$.
\end{proposition}

\begin{proposition}\label{P4.6}
A nonempty subset $Q$ of an $AG$-groupoid $S$ is its quasi-ideal if and only if $\left(C_{Q}\right)_{k}$ is an $(\in,\in\!\vee q_{k})$-fuzzy quasi-ideal of $S$.
\end{proposition}

\begin{corollary}\label{C4.7}
Any $(\in,\in \vee q_{k})$-fuzzy left ideal of an $AG$-groupoid is its
$(\in,\in \vee q_{k})$-fuzzy quasi-ideal.
\end{corollary}

\begin{theorem}\label{T4.8}
For a unitary $AG$-groupoid the following are equivalent.

$\;\;\,(i)$ \ $S$ is intra-regular.

$\;\,(ii)$  \ $Sa\subseteq Sa^2$ for every $a\in S$.

$(iii)$ \ $I\cap J\subseteq IJ$ for all quasi-ideals of $S$.

$\,(iv)$ \ $f\wedge_{k}g\leq f\circ_{k}g$ for all $(\in,\in \vee q)$-fuzzy quasi-ideals of $S$.
\end{theorem}
\begin{proof}
$(i)\Longrightarrow(iv).$ \ Let $f$ and $g$ be $(\in,\in \vee q)$-fuzzy quasi-ideals of a unitary intra-regular $AG$-groupoid $S$. Then by Lemma \ref{L3.4}, 
$$
a=a\cdot za \ \ \ \ {\rm and } \ \ \ \ wa=a\cdot (w\cdot za)
$$ 
because $wa=ew\cdot(a\cdot za)=(za\cdot a)\cdot we=ea\cdot (w\cdot za)=a\cdot (w\cdot za)$. Hence,
\begin{align*}
\left(f\circ_{k}g\right)(a)&\geq \Big(\Big((S\circ f)\wedge_k(f\circ S)\Big)\circ_k\Big((S\circ g)\wedge_k(g\circ S)\Big)\Big)(a)\\
&={\displaystyle\bigvee\limits_{a=pq}}
\left\{\Big((S\circ f)\wedge_k(f\circ S)\Big)(p)\wedge_k\Big((S\circ g)\wedge_k(g\circ S)\Big)(q)\right\} \\
&={\displaystyle\bigvee\limits_{a=pq}}
\Big\{(S\circ f)(p)\wedge_k(f\circ S)(p)\wedge_k(S\circ g)(q)\wedge_k(g\circ S)(q)\Big\} \\
&\geq{\displaystyle\bigvee\limits_{a=wa\cdot a}}
\Big\{(S\circ f)(wa)\wedge_k(f\circ S)(wa)\wedge_k(S\circ g)(a)\wedge_k(g\circ S)(a)\Big\} \\ 
&={\displaystyle\bigvee\limits_{a=wa\cdot a}}
\Big\{(S\circ f)(wa)\wedge_k(f\circ S)(a(w\cdot za))\wedge_k(S\circ g)(ea)\wedge_k(g\circ S)(a\cdot za)\Big\} \\
&=\!\!\!{\displaystyle\bigvee\limits_{a=wa\cdot a}}
\Big\{(S(w)\!\wedge\! f(a)\!\wedge_k\! f(a)\!\wedge S(w\cdot za)\!\wedge_k\! S(e)\!\wedge\! g(a)\!\wedge_k\! g(a)\!\wedge \!S(a\cdot za)\Big\} \\
&=f(a)\wedge_k g(a)=(f\wedge_k g)(a) .
\end{align*}
Therefore $f\circ_{k}g\geq f\wedge_{k}g$.

\medskip 

$(iv)\Longrightarrow(iii).$ \ Let $a\in I\cap J$, where $I$ and $J$ are quasi-ideals. Then 
$$
(C_{IJ})_{k}(a)=(C_{I}\circ_{k}C_{J})(a)\geq (C_{I}\wedge_{k}C_{J})(a)=(C_{I\cap J})_{k}(a)\geq\frac{1-k}{2}
$$
by Lemma \ref{L2.7}. Thus, $a\in IJ$. Consequently, $I\cap J\subseteq IJ$ for any quasi-ideals of $S$.

\medskip

$(iii)\Longrightarrow(ii).$ By Proposition \ref{P4.1} and Corollary \ref{C4.2}.

\medskip

$(ii)\Longrightarrow(i).$ \ By Theorem \ref{T3.8}.
\end{proof}

\begin{corollary}\label{C3.4}
A unitary $AG$-groupoid $S$ is intra-regular if and only if one of the following conditions is satisfied;

\begin{enumerate}
\item[$(i)$] $Q^2=Q$ for all quasi-ideals of $S$,
\item[$(ii)$] $Sa=Sa^2$ for every $a\in S$,
\item[$(iii)$] $f_k\leq f\circ_k f$ for all $(\in,\in \vee q_{k})$-fuzzy quasi-ideals of $S$.
\end{enumerate}
\end{corollary}
\begin{proof}
By Theorem \ref{T4.8} in unitary intra-regular $AG$-groupoid $S$ for every quasi-ideal $Q$ we have $Q\subseteq Q^2$. On the other hand, $Q^2=Q^2\cap Q^2\subseteq SQ\cap QS\subseteq Q$. So, $Q^2=Q$. Similarly, 
$Sa\subseteq Sa^2=Sa\cdot Sa\subseteq S^2\cdot Sa=aS\cdot S^2=Sa$ by Corollary \ref{C4.2}.
 
 The rest is clear.
 \end{proof}
\section{Conclusions}

Unitary intra-regular $AG$-groupoids can be characterized by the properties of their left ideals, bi-ideals and quasi-ideals. A crucial role in this characterization play ideals of the form $Sa$ and $Sa^2$. Seems that future research work should focus on understanding the role played by subsets $Sa$ and $Sa^2$ in the theory of all $AG$-groupoids and not just in intra-regular $AG$-groupoids.

\end{document}